\documentclass[11pt]{amsart}
\advance\hoffset by -0.5 truecm
\usepackage{amsmath,amscd}
\usepackage{amssymb}
\usepackage{amsthm}
\usepackage{array}
\usepackage{graphicx}
\usepackage{color}
\usepackage[colorinlistoftodos]{todonotes}
\usepackage[T1]{fontenc}
\usepackage[utf8]{inputenc}
\usepackage[french,english]{babel}

\makeatletter
\newenvironment{abstracts}{%
  \ifx\maketitle\relax
    \ClassWarning{\@classname}{Abstract should precede
      \protect\maketitle\space in AMS document classes; reported}%
  \fi
  \global\setbox\abstractbox=\vtop \bgroup
    \normalfont\Small
    \list{}{\labelwidth\z@
      \leftmargin3pc \rightmargin\leftmargin
      \listparindent\normalparindent \itemindent\z@
      \parsep\z@ \@plus\p@
      
      \itemsep\medskipamount
    }%
}{%
  \endlist\egroup
  \ifx\@setabstract\relax \@setabstracta \fi
}

\newcommand{\abstractin}[1]{%
  \otherlanguage{#1}%
  \item[\hskip\labelsep\scshape\abstractname.]%
}
\makeatother

\usepackage{mathrsfs}
\usepackage{hyperref}

\usepackage{comment}

\usepackage{calc}
             {\begin{list}{\arabic{enumi}.}{\usecounter{enumi}%
              \setlength{\labelsep}{0.5em}%
              \settowidth{\labelwidth}{\arabic{enumi}.}%
              \setlength{\leftmargin}{\labelwidth+\labelsep}}}%
             {\end{list}}

\usepackage{mathtools}

\usepackage{amsmath, amsthm, amssymb, enumerate, hyperref, xcolor, tikz-cd, listings,hyperref,comment}

\newcommand{\QQ}{\mathcal{Q}}

\newcommand{\CC}{\mathcal{C}}
\newcommand{\R}{\mathbb{R}}
\newcommand{\B}{\mathcal{B}}
\newcommand{\M}{\mathcal{M}}

\newcommand{\MLS}{\mathrm{MLS}}

\newcommand{\var}{\mathrm{Var}}

\renewcommand{\hat}{\widehat}

\newcommand{\NN}{\mathbb{N}}

\newcommand{\PP}{\mathcal{P}}

\renewcommand{\phi}{\varphi}

\renewcommand{\var}{\mathrm{Var}}

\renewcommand{\var}{\mathrm{Var}}

\def\semicolon{;}
\def\applytolist#1{
    \expandafter\def\csname multi#1\endcsname##1{
        \def\multiack{##1}\ifx\multiack\semicolon
            \def\next{\relax}
        \else
            \csname #1\endcsname{##1}
            \def\next{\csname multi#1\endcsname}
        \fi
        \next}
    \csname multi#1\endcsname}

\def\calc#1{\expandafter\def\csname c#1\endcsname{{\mathcal #1}}}
\applytolist{calc}QWERTYUIOPLKJHGFDSAZXCVBNM;
\def\bbc#1{\expandafter\def\csname bb#1\endcsname{{\mathbb #1}}}
\applytolist{bbc}QWERTYUIOPLKJHGFDSAZXCVBNM;
\def\bfc#1{\expandafter\def\csname bf#1\endcsname{{\mathbf #1}}}
\applytolist{bfc}QWERTYUIOPLKJHGFDSAZXCVBNM;
\def\sfc#1{\expandafter\def\csname s#1\endcsname{{\sf #1}}}
\applytolist{sfc}QWERTYUIOPLKJHGFDSAZXCVBNM;
\def\fc#1{\expandafter\def\csname f#1\endcsname{{\mathfrak #1}}}
\applytolist{fc}QWERTYUIOPLKJHGFDSAZXCVBNM;
\theoremstyle{plain}
\newtheorem{theorem}{Theorem}[section]

\newtheorem*{thm*}{Theorem}
\newtheorem{corollary}[theorem]{Corollary}
\newtheorem*{corollary*}{Corollary}
\newtheorem{lemma}[theorem]{Lemma}
\newtheorem*{lemma*}{Lemma}
\newtheorem{proposition}[theorem]{Proposition}
\newtheorem*{proposition*}{Proposition}

\theoremstyle{definition}

\newtheorem*{defn*}{Definition}

\newtheorem*{Ex*}{Example}

\renewcommand{\Im}{\mathrm{Im}}
\renewcommand{\Re}{\mathrm{Re}}

\DeclareRobustCommand{\SkipTocEntry}[5]{}

\newcounter{sidenote}
\begin{document}

\title[A Positive Proportion Livshits Theorem]{A Positive Proportion Livshits Theorem}

\author[C. Dilsavor]{Caleb Dilsavor}
\address[Caleb Dilsavor]{The Ohio State University, 231 West 18$^{\text{th}}$ Avenue, 43210, Columbus, OH, USA}
\email{\href{mailto:dilsavor.4@osu.edu}{dilsavor.4@osu.edu}}

\author[J. Marshall Reber]{James Marshall Reber}
\address[James Marshall Reber]{The Ohio State University, 231 West 18$^{\text{th}}$ Avenue, 43210, Columbus, OH, USA}
\email{\href{mailto:marshallreber.1@osu.edu}{marshallreber.1@osu.edu}}

\date{\today}

\begin{abstracts}
\abstractin{english}
Given a transitive Anosov diffeomorphism or flow on a closed connected Riemannian manifold $M$, the Livshits theorem states that a H\"{o}lder function $\phi : M \to \R$ is a coboundary if all of its periods vanish. We explain how a finer statistical understanding of the distribution of these periods can be used to obtain a stronger version of the classical Livshits theorem where one only has to check that the periods of $\phi$ vanish on a set of positive asymptotic upper density. We also include a strengthening of the nonpositive Livshits theorem.
\end{abstracts}

\thanks{This work was partially supported by NSF grant DMS-$1954463$.}

\maketitle

\section{Introduction}
Let $M$ be a closed connected Riemannian manifold and let $f^t : M \rightarrow M$ be a transitive Anosov diffeomorphism or flow.  
If $\phi_1, \, \phi_2 : M \to \R$ are H\"{o}lder, we say that $\phi_1$ and $\phi_2$ are \emph{cohomologous} if there is a H\"{o}lder continuous $\kappa : M \rightarrow \R$ such that 
\begin{align*} 
\phi_1-\phi_2 &= \begin{dcases} \kappa \circ f - \kappa &\text{if } f \text{ is a diffeomorphism}, \\ 
\left.\frac{d}{dt}\right|_{t=0} (\kappa \circ f^t) & \text{if } f^t \text{ is a flow}.
\end{dcases} 
\end{align*}
We write $\phi_1 \sim \phi_2$ if they are cohomologous, and we say that $\phi$ is a \emph{coboundary} if $\phi \sim 0$.

Let $P$ be the collection of all closed orbits of $f^t$. For each $\gamma \in P$, we denote its length by $\ell(\gamma)$ and let $x_\gamma \in \gamma$ be an arbitrary point. Given a H\"{o}lder continuous function $\phi : M \rightarrow \R$, its \emph{period} $\ell_\phi(\gamma)$ along a closed orbit $\gamma$ is given by\[
    \ell_\phi(\gamma) \coloneqq \begin{dcases}\sum_{k=0}^{\ell(\gamma)-1} \phi(f^k(x_\gamma)) & \text{if } f \text{ is a diffeomorphism}, \\
    \int_0^{\ell(\gamma)} \phi(f^s(x_\gamma))\,ds & \text{if } f^t \text{ is a flow}.
    \end{dcases}
\]
If $\phi$ is a coboundary then all of its periods $\ell_\phi(\gamma)$ must vanish, and if $\phi$ is cohomologous to a nonpositive function then all of its periods must be nonpositive. The converses also hold: this is the content of the Livshits theorem \cite{Livshits} and the nonpositive Livshits theorem \cite{lopes1, lopes2}. 

We will show that the hypothesis in the Livshits theorem only needs to be verified on a subset of $P$ that has positive asymptotic upper density with respect to some H\"{o}lder weight function $\psi$. In the unweighted diffeomorphism case, for example, our result states
\[ \limsup_{n \rightarrow \infty} \frac{|\{\gamma \in P \ | \ \ell_\phi(\gamma) = 0 \text{ and } \ell(\gamma) = n\}|}{|\{\gamma \in P \ | \ \ell(\gamma) = n \}|} > 0 \implies \phi \sim 0.\]
Allowing for a weight function $\psi$ generalizes this considerably and requires a new weighted version of the Cantrell-Sharp central limit theorem in \cite{cantrellsharp}. See Theorem \ref{thm:CCLT} and Section \ref{sec:CLTs}.

The motivation to prove this result originated with the recent rigidity work of Gogolev and Rodriguez Hertz \cite[Theorem 1.1]{GRH}. The positive proportion Livshits theorem is used in their proof that a $C^0$ conjugacy between $C^r$ ($r > 2$) volume preserving Anosov flows on a $3$-dimensional manifold must be $C^{r_*}$ if one of the flows is not a constant roof suspension, where 
\[ r_* = \begin{dcases} r &\text{if } r \notin \NN \\ r-1 + \text{Lip} &\text{if }r \in \NN. \end{dcases}\]
In the application to their setting, the weight function $\psi$ is taken to be the geometric potential, which we remark generally cannot be reduced to the unweighted case.

In order to state the main results, for a subset $A \subseteq P$ we let $A(n)$ denote the elements of $A$ with length $n$ and $A(T,\Delta)$ the elements of $A$ whose length lies in $(T,T+\Delta]$.

\begin{theorem} \label{thm:positivepropLivshits1}
Let $f : M \rightarrow M$ be a transitive Anosov diffeomorphism, $\phi: M \rightarrow \R$ H\"{o}lder continuous, and $Q \coloneqq \{\gamma \in P \mid \ell_\varphi(\gamma) = 0\}.$
If there is a H\"{o}lder continuous function $\psi : M \rightarrow \R$ such that
\[ \limsup_{n \rightarrow \infty}  \frac{\sum_{\gamma \in Q(n)} \exp \left( \ell_\psi(\gamma) \right)}{\sum_{\gamma \in P(n)} \exp \left( \ell_\psi(\gamma) \right)} > 0,\]
then $\phi$ is a coboundary.
\end{theorem}

\begin{theorem} \label{thm:positivepropLivshits2}
Let $f^t : M \rightarrow M$ be a transitive Anosov flow whose stable and unstable distributions are not jointly integrable,
$\phi : M \rightarrow \R$ H\"{o}lder continuous, $\Delta > 0$, and
${Q \coloneqq \{\gamma \in P \mid \ell_\phi(\gamma) = 0\}}$.
If there is a H\"{o}lder continuous function $\psi : M \rightarrow \R$ such that
\[ \limsup_{T \rightarrow \infty} \frac{\sum_{\gamma \in Q(T,\Delta)} \exp \left( \ell_\psi(\gamma) \right)}{\sum_{\gamma \in P(T,\Delta)} \exp \left( \ell_\psi(\gamma) \right)} > 0,\]
then $\phi$ is a coboundary. If the topological pressure of $\psi$ is positive, then the same statement holds for $P(0,T)$ and $Q(0,T)$ in place of $P(T,\Delta)$ and $Q(T,\Delta)$, respectively.
\end{theorem}

While our primary motivation is to show these generalizations for Anosov systems, we remark that Theorems \ref{thm:positivepropLivshits1} and \ref{thm:positivepropLivshits2} 
also hold for Axiom A diffeomorphisms and hyperbolic flows satisfying an approximability condition given in \cite{pollicottsharp}.
The extra hypothesis in the flow case comes from an application of Dolgopyat's bounds on iterates of the transfer operator \cite{dolgopyat1998prevalence, pollicottsharp}. 
If the Plante conjecture is true, then the only Anosov flows to which Theorem \ref{thm:positivepropLivshits2} does not apply are conjugate to constant roof suspensions of Anosov diffeomorphisms, and one can apply Theorem \ref{thm:positivepropLivshits1} instead. The Plante conjecture is known for 3-dimensional volume preserving flows \cite{plante},
and hence this is not an issue in \cite{GRH}.

Since it is natural to wonder what can be said for the nonpositive Livshits theorem, we show a strengthening of this as well.

\begin{theorem} \label{thm:positivepropLivshits3}
Let $f : M \rightarrow M$ be a transitive Anosov diffeomorphism, $\phi : M \rightarrow \R$ H\"{o}lder continuous, and $Q' \coloneqq \{\gamma \in P \mid \ell_\phi(\gamma) > 0\}$.
If 
\[\lim_{n\to\infty}\frac{1}{n}\log|Q'(n)| = 0,\]
then $\phi$ is cohomologous to a nonpositive function.
\end{theorem}

\begin{theorem} \label{thm:positivepropLivshits4}
Let $f^t : M \rightarrow M$ be a transitive Anosov flow whose stable and unstable foliations are not jointly integrable,
$\phi : M \rightarrow \R$ H\"{o}lder continuous, $\Delta > 0$, and ${Q' \coloneqq \{\gamma \in P \mid \ell_\phi(\gamma) > 0\}}$.
If 
\[\lim_{T\to\infty}\frac{1}{T}\log|Q'(T, \Delta)| = 0,\]
then $\phi$ is cohomologous to a nonpositive function.
The same statement holds for $Q'(0,T)$ in place of $Q'(T,\Delta)$.
\end{theorem}

A useful example to keep in mind for Theorems \ref{thm:positivepropLivshits3} and \ref{thm:positivepropLivshits4} is the full one-sided shift on two symbols $\Sigma = \{0,1\}^\NN$ with observable $\phi(x) \coloneqq (-1)^{x_1}$.
Clearly $\phi$ has a positive period.
However, adding weights according to the function
\[
\psi(x) = \begin{cases} \log p, & x_1 = 0 \\ \log(1-p) & x_1 = 1  \end{cases}
\]
where $p \in (0,\frac{1}{2})$, we see that $Q'$ has zero asymptotic density with respect to $\psi$ by the law of large numbers for the $(p,1-p)$ coin flip. Thus $\ell_\phi$ is nonpositive on a set of full proportion with respect to $\psi$, even though $\phi$ is not cohomologous to a nonpositive function.

We end this section with an application of Theorems \ref{thm:positivepropLivshits2} and \ref{thm:positivepropLivshits4} to partial marked length spectrum rigidity. Let $M$ be a closed surface and $g$ a Riemannian metric on $M$ with everywhere negative sectional curvature. Recall that there is a unique closed geodesic of minimal length in every free homotopy class. If we denote the set of free homotopy classes by $\CC$ and the length of a curve $\gamma$ with respect to the metric $g$ by $\ell_g(\gamma)$, then the marked length spectrum is the function 
\[ \MLS_g : \CC \rightarrow \R, \ \ \MLS_g(\sigma) \coloneqq \ell_g(\gamma), \]
where $\gamma \in \sigma$ is the unique $g$-geodesic representative.

It was shown by Sawyer \cite{sawyer} that if $g_1$ and $g_2$ are two negatively curved metrics on $M$ such that $\MLS_{g_1} \geq \MLS_{g_2}$ (respectively $\MLS_{g_1} = \MLS_{g_2})$ on a set of free homotopy classes whose complement has subexponential growth, then $\MLS_{g_1} \geq \MLS_{g_2}$  (respectively $\MLS_{g_1} = \MLS_{g_2})$ on all of $\CC$. Therefore, with this weakened assumption, one still has $\text{Area}(g_1) \geq \text{Area}(g_2)$, with equality if and only if $g_1$ and $g_2$ are isometric. 
We are able to recover and partially improve on these results with Theorems \ref{thm:positivepropLivshits2} and \ref{thm:positivepropLivshits4}.

\begin{theorem} \label{thm:pmls}
Let $M$ be a closed surface, let $g_1, g_2$ be two negatively curved metrics on $M$, and let $\CC$ be the set of free homotopy classes. Define
\[\begin{split} \CC(T) \coloneqq \{ [\gamma] \in \CC \ | \ \MLS_{g_1}(\gamma) \leq T\}, \ \ \B(T) \coloneqq \{ [\gamma] \in \CC(T) \ | \ \MLS_{g_1}(\gamma) = \MLS_{g_2}(\gamma) \},\\ \QQ(T) \coloneqq \{ [\gamma] \in \CC(T) \ | \ \MLS_{g_1}(\gamma) < \MLS_{g_2}(\gamma) \}.\end{split}\]
\begin{enumerate}[(a)]
    \item If $\QQ(T)$ grows subexponentially, then $\text{Area}(g_1) \geq \text{Area}(g_2)$.
    \item If $\limsup_{T \rightarrow \infty} (\# \B(T)/\# \CC(T)) > 0,$
    then $g_1$ and $g_2$ are isometric.
\end{enumerate}
\end{theorem}

\begin{proof}
Let ${f_1^t : S_{g_1}M \rightarrow S_{g_1}M}$ and $f_2^t : S_{g_2}M \rightarrow S_{g_2}M$ be the corresponding geodesic flows. As noted in \cite{gromov}, there is an orbit equivalence $h : S_{g_1}M \rightarrow S_{g_2}M$ which is homotopic to the identity. By \cite[Theorem 19.1.5]{KH}, we may assume that the orbit equivalence $h$ is H\"{o}lder continuous and differentiable along $f_1^t$. Let $\alpha$ be the cocycle over $f_1^t$ satisfying 
    \[ h(f_1^t(v)) = f_2^{\alpha(v,t)}(h(v)),\]
and let 
    \[\phi : S_{g_1}M \rightarrow \R, \ \ \phi(v) \coloneqq \left.\frac{d}{dt}\right|_{t=0} \alpha(v,t) - 1.\]
If $\sigma \in \CC$ and $\gamma \in \sigma$ is the $g_1$-geodesic representative, then we have
\[ \ell_\phi(\gamma) = \MLS_{g_2}(\sigma) - \MLS_{g_1}(\sigma).\]
It is clear that (a) follows from Theorem \ref{thm:positivepropLivshits4} and \cite[Theorem 1.1]{croke2004lengths}, while (b) follows from Theorem \ref{thm:positivepropLivshits2} and \cite[Theorem 1]{otal}.
\end{proof}

\addtocontents{toc}{\SkipTocEntry}
\subsection*{Acknowledgements}
We thank Andrey Gogolev for suggesting the problem and for reading an early draft, Daniel Thompson for advice and for reading an early draft, and Thomas O'Hare and Stephen Cantrell for useful discussions.

\section{Preliminaries} \label{sec:preliminaries}
We start by collecting basic definitions and thermodynamic results. A $C^1$ diffeomorphism $f : M \rightarrow M$ is \emph{Anosov} 
if the following conditions are satisfied:
\begin{enumerate}
    \item There is a splitting of the tangent bundle $TM = E^s \oplus E^u$ and constants $C > 0$ and $\lambda \in (0,1)$ such that 
        \[ \|D_xf^n(v)\| \leq C \lambda^n \text{ for all } v \in E^s(x), \ n \geq 0,\]
        \[ \|D_xf^{-n}(v)\| \leq C \lambda^n \text{ for all } v \in E^u(x), \ n \geq 0.\]
    \item We have 
    $ Df(E^u) = E^u$ and $Df(E^s) = E^s.$
\end{enumerate}
A $C^1$ flow $f^t : M \rightarrow M$ is \emph{Anosov} if the following conditions are satisfied:
\begin{enumerate}
    \item There is a splitting of the tangent bundle $TM = E^s \oplus E^0 \oplus E^u$ and constants $C > 0$ and $\lambda \in (0,1)$ such that 
        \[ \|D_xf^t(v)\| \leq C \lambda^n \text{ for all } v \in E^s(x), \ t \geq 0,\]
        \[ \|D_xf^{-t}(v)\| \leq C \lambda^n \text{ for all } v \in E^u(x), \ t \geq 0.\]
    \item The distribution $E^0$ is one-dimensional and tangent to the flow.
    \item We have $Df^t(E^u) = E^u$ and $Df^t(E^s) = E^s$ for all $t \in \R.$
\end{enumerate}

Let $f^t$ denote either a transitive Anosov diffeomorphism or a transitive Anosov flow which is weak mixing. In particular, the flow is weak mixing if $E^s \oplus E^u$ is not integrable.
Let $\M_f$ denote the set of $f^t$-invariant Borel probability measures. Given a H\"{o}lder weight function $\psi : M \to \R$, the \emph{topological pressure} $\PP(\psi)$ is
\begin{equation}\label{eq:pressuredef}
\PP(\psi) \coloneqq \sup_{\mu\in \M_f}\left\{h(\mu)+\mu(\psi)\right\},
\end{equation}
where $h(\mu)$ denotes the metric entropy of $\mu$ with respect to $f^1$, and $\mu(\psi)$ denotes the integral of $\psi$ with respect to $\mu$.
This can be computed as
\begin{equation} \label{eqn:pressure}
\PP(\psi) = \lim_{n\to\infty}\frac{1}{n}\log\left(\sum_{\gamma\in P(n)}\exp(\ell_\psi(\gamma))\right)
\end{equation}
for discrete time, and similarly with $P(T,\Delta)$ for continuous time. Equation \eqref{eqn:pressure} also holds for $P(0,T)$ under the assumption that $\PP(\psi) \geq 0$.

A measure achieving the supremum in \eqref{eq:pressuredef} is called an \emph{equilibrium state} for $\psi$. The notion of an equilibrium state for $\psi$ is unchanged if a constant is added to $\psi$ or if $\psi$ is replaced by a cohomologous function. It is a classical result that $\psi$ has a unique equilibrium state $\mu_\psi$ given by
\[
\mu_\psi = \lim_{n\to\infty}\frac{\sum_{\gamma\in P(n)}\exp(\ell_\psi(\gamma))\delta_\gamma}{\sum_{\gamma\in P(n)}\exp(\ell_\psi(\gamma))}
\]
in the diffeomorphism case, where $\delta_\gamma$ denotes the $f^t$-invariant probability measure supported on $\gamma$ \cite{bowen}. The same holds for flows if one replaces $P(n)$ with $P(T,\Delta)$ and assumes $\PP(\psi) \geq 0$  
\cite[Chapter 7]{PP}. 
Equilibrium states in this setting have positive entropy: this can be deduced from \cite[Lemma 2]{parry}.

Given a H\"{o}lder observable $\phi : M \rightarrow \R$ and $\mu \in \M_f$, the \emph{variance} of $\phi$ with respect to $\mu$ is 
\[ \var_{\mu}(\phi) \coloneqq \mu\left( (\phi - \mu(\phi))^2 \right).\]
The \emph{dynamical variance} of $\phi$ with respect to $\mu$ is then
    \[ \sigma^2_{\phi,\mu} \coloneqq \lim_{n \rightarrow \infty} \frac{1}{n}\var_{\mu}(S_n(\phi)), \text{ where } S_n(\phi)(x) \coloneqq \sum_{k=0}^{n-1} \phi(f^k(x))\]
in discrete time, or 
\[ \sigma^2_{\phi,\mu} \coloneqq \lim_{T \rightarrow \infty} \frac{1}{T}\var_{\mu}(S_T(\phi)), \text{ where } S_T(\phi)(x) \coloneqq \int_0^T \phi(f^t(x))\,dt\]
in continuous time.
If $\psi: M \to \R$ is H\"{o}lder, then we let $\sigma^2_{\phi,\psi} \coloneqq \sigma^2_{\phi,\mu_\psi}$, where $\mu_\psi$ is the equilibrium state for $\psi$. Adding constants to $\phi,\psi$ does not change the value of $\sigma^2_{\phi,\psi}$, nor does replacing $\phi, \psi$ with cohomologous functions. In terms of these quantities, the pressure function has the expansion 
\begin{equation} \label{eqn:pressuretaylor}
\PP(\psi+t\phi) = \PP(\psi) + \mu_\psi(\phi)t + \frac{\sigma_{\phi,\psi}^2 }{2}t^2 + o(t^2).
\end{equation}
If $\phi$ is cohomologous to a constant, then
$\sigma_{\phi,\psi}^2 = 0$ for all $\psi$. Conversely, if for even a single H\"{o}lder $\psi$ one has $\sigma_{\phi,\psi}^2 = 0$, then $\phi$ is cohomologous to a constant \cite{PP,ratner2}.

\section{Proofs of the main theorems}

We begin by demonstrating that the main theorems are unaffected if the non-prime orbits are excluded. This will also be useful for obtaining counting estimates in Section \ref{sec:CLTs}. Let $P_* \subseteq P$ denote the collection of prime closed orbits, where a closed orbit $\gamma$ is called \emph{prime} if $\gamma = (\gamma')^n$ for some closed orbit $\gamma'$ implies $\gamma = \gamma'$. Let $Q_*$ and $Q'_*$ be the subsets of $P_*$ defined as in the main theorems. To show Theorems \ref{thm:positivepropLivshits3} and \ref{thm:positivepropLivshits4} are equivalently stated for prime orbits, we simply note that
\[
|Q'_*(n)| \leq |Q'(n)| \leq \sum_{k=1}^n |Q'_*(k)|.
\]
Thus $|Q'_*(n)|$ and $|Q'(n)|$ have the same exponential growth rate. The same argument works for the flow case as well.

As for Theorems \ref{thm:positivepropLivshits1} and \ref{thm:positivepropLivshits2}, it suffices to know
\[
\limsup_{n \rightarrow \infty}  \frac{\sum_{\gamma \in Q(n)} \exp \left( \ell_\psi(\gamma) \right)}{\sum_{\gamma \in P(n)} \exp \left( \ell_\psi(\gamma) \right)} = \limsup_{n \rightarrow \infty}  \frac{\sum_{\gamma \in Q_*(n)} \exp \left( \ell_\psi(\gamma) \right)}{\sum_{\gamma \in P_*(n)} \exp \left( \ell_\psi(\gamma) \right)},
\]
along with the corresponding statements for the flow case using either $Q(T,\Delta)$ and $P(T,\Delta)$ or $Q(0,T)$ and $P(0,T)$.
These statements are clear from the following lemmas, whose proofs follow from \cite[pp.\ 115-116]{PP}.
\begin{lemma}\label{lem:primediscrete}
     Assume we are in the diffeomorphism case. Then for any H\"{o}lder $\psi : M \to \R$, we have
     \[\limsup_{n\to\infty}\frac{1}{n}\log\left(\sum_{\gamma \in P(n)\setminus P_*(n)} \exp(\ell_\psi(\gamma))\right) < \PP(\psi). \]
\end{lemma}
\begin{lemma}\label{lem:prime}
    Assume we are in the flow case. Then for any H\"{o}lder $\psi : M \to \R$, we have
    \[
    \limsup_{T\to\infty}\frac{1}{T}\log\left(\sum_{\gamma \in P(T,\Delta)\setminus P_*(T,\Delta)} \exp(\ell_\psi(\gamma))\right) < \PP(\psi). 
    \]
    Furthermore, if $\PP(\psi) > 0$, then
    \[
    \limsup_{T\to\infty}\frac{1}{T}\log\left(\sum_{\gamma\in P(0,T)\setminus P_*(0,T)}\exp(\ell_\psi(\gamma))\right) < \PP(\psi).
    \]
    \end{lemma}

\subsection{Proof of Theorems \ref{thm:positivepropLivshits1} and \ref{thm:positivepropLivshits2}}
Both of these results follow from central limit theorems describing the asymptotic distribution of the numbers $\ell_\phi(\gamma)$ as $\ell(\gamma) \to \infty$. 
The main point is that, because $\sigma_{\phi,\psi}$ is the standard deviation of the limiting distribution in these theorems, there would be too much mass concentrated in one place for $\sigma_{\phi,\psi}$ to be nonzero if the positive proportion hypothesis is true. 
This implies that $\phi$ is cohomologous to a constant, which then has to be zero. We remark that, unlike the theorem of Ratner \cite{ratner2}, the fact that $\sigma_{\phi,\psi}$ is the correct standard deviation in the following theorems is nontrivial by itself and does not follow from direct computation.

\begin{theorem}[{\cite[Theorem 5]{coelho1990central}}] \label{thm:discreteCLT}
    Let $f : M \rightarrow M$ be a transitive Anosov diffeomorphism and let $\phi, \psi : M \rightarrow \R$ be H\"{o}lder continuous. Suppose that $\sigma_{\phi,\psi}^2 > 0$ and $\mu_\psi(\phi) = 0$. Let $\mu_{n,\psi}$ be the probability measure on $P(n)$ given by
    \[ 
    \mu_{n, \psi} \coloneqq \frac{\sum_{P(n)} \exp(\ell_\psi(\gamma)) \delta_\gamma}{\sum_{P(n)} \exp(\ell_\psi(\gamma))}.\]
    Then $\ell_\phi/\sqrt{n}$, considered as a random variable on $P(n)$, converges in distribution to a normal distribution with mean zero and standard deviation $\sigma_{\phi, \psi}$ as $n \to \infty$. 
\end{theorem}

\begin{theorem}\label{thm:CCLT}
Let $f^t : M \rightarrow M$ be a transitive Anosov flow whose stable and unstable distributions are not jointly integrable,
let ${\phi, \psi : M \rightarrow \R}$ be H\"{o}lder continuous with ${\PP(\psi) > 0}$, and let $\Delta > 0$. Suppose that $\sigma_{\phi,\psi}^2 > 0$ and $\mu_\psi(\phi) = 0$. Let $\mu_{T,\Delta,\psi}$ be the probability measure on $P(T,\Delta)$ given by
    \[ 
    \mu_{T,\Delta,\psi} \coloneqq \frac{\sum_{P(T,\Delta)} \exp(\ell_\psi(\gamma)) \delta_\gamma}{\sum_{P(T,\Delta)} \exp(\ell_\psi(\gamma))}.  \]
    Then $\ell_\phi/\sqrt{T}$, considered as a random variable on $P(T,\Delta)$, converges in distribution to a normal distribution with mean zero and standard deviation $\sigma_{\phi,\psi}$ as $T\to\infty$. 
\end{theorem}
Theorem \ref{thm:CCLT} was proved for $\psi \equiv 0$ in \cite{cantrellsharp}. For $\psi \not\equiv 0$, the strategy carries through without major modifications as long as $\PP(\psi) > 0$. We will give the details for the weighted case in Section \ref{sec:CLTs}.

\begin{proof}[Proof of Theorem \ref{thm:positivepropLivshits1}]
    Assume the positive proportion hypothesis is true, and assume for contradiction that $\sigma_{\phi,\psi}^2 > 0$. 
    Let $\phi' = \phi-\mu_\psi(\phi)$, and
    let $X_n$ be the random variable on $P(n)$ given by
    \[
    X_n = \frac{\ell_{\phi'}}{\sqrt{n}} = \frac{\ell_{\phi} - n\mu_\psi(\phi)}{\sqrt{n}}.
    \]
    By Theorem \ref{thm:discreteCLT}, we know that $X_n$ converges in distribution to a normal distribution with mean zero and standard deviation $\sigma_{\phi',\psi} = \sigma_{\phi,\psi} \neq 0$. In terms of $X_n$, the positive proportion hypothesis states
    \[\limsup_{n\to\infty}\mu_{n,\psi}(X_n = -\mu_\psi(\phi)\sqrt{n}) > 0.\] If $\mu_\psi(\phi) \neq 0$, then this contradicts tightness of $X_n$. If $\mu_\psi(\phi) = 0$, then we have a contradiction to the fact that the limiting distribution of $X_n$ has no atom at $0$, and therefore
    \[\lim_{\epsilon\to 0}\limsup_{n\to\infty}\mu_{n,\psi}(X_n \in [-\epsilon,\epsilon]) = 0.\]
    Therefore $\sigma^2_{\phi,\psi} = 0$, and so $\phi$ is cohomologous to a constant. Since $\phi$ has a zero period by hypothesis, it must be a coboundary.
\end{proof}

\begin{proof}[Proof of Theorem \ref{thm:positivepropLivshits2}]
Replacing Theorem \ref{thm:discreteCLT} with Theorem \ref{thm:CCLT}, the proof of Theorem \ref{thm:positivepropLivshits2} for $P(T,\Delta)$ and $Q(T,\Delta)$ is identical to the above if $\PP(\psi) > 0$. If $\PP(\psi) \leq 0$, then choose $c$ large enough so that ${\PP(\psi + c) > 0}$, and let $\psi' = \psi+c$. Observe
\[  \frac{\sum_{\gamma \in Q(T,\Delta)} \exp \left( \ell_{\psi'}(\gamma) \right)}{\sum_{\gamma \in P(T,\Delta)} \exp \left( \ell_{\psi'}(\gamma) \right)} \geq e^{-c \Delta} \frac{\sum_{\gamma \in Q(T,\Delta)} \exp \left( \ell_\psi(\gamma) \right)}{\sum_{\gamma \in P(T,\Delta)} \exp \left( \ell_\psi(\gamma) \right)}
,  \]
and therefore if $Q$ has positive proportion with respect to $\psi$, it has positive proportion with respect to $\psi'$. Hence $\phi$ is a coboundary using the same argument for $\psi'$ in place of $\psi$.

As for the last part of Theorem \ref{thm:positivepropLivshits2}, suppose that $\PP(\psi) > 0$ and that
\[
\limsup_{T\to\infty} \frac{\sum_{\gamma \in Q(0,T)}\exp(\ell_\psi(\gamma))}{\sum_{\gamma\in P(0,T)}\exp(\ell_\psi(\gamma))} 
= \epsilon > 0.
\]
For any $\Delta > 0$, we may write
\begin{align*}
    \frac{\sum_{\gamma \in Q(T,\Delta)}\exp(\ell_\psi(\gamma))}{\sum_{\gamma\in P(T,\Delta)}\exp(\ell_\psi(\gamma))}
    & \geq \frac{\sum_{\gamma \in Q(T,\Delta)}\exp(\ell_\psi(\gamma))}{\sum_{\gamma\in P(0,T+\Delta)}\exp(\ell_\psi(\gamma))} \\
    & \geq \frac{\sum_{\gamma \in Q(0,T+\Delta)}\exp(\ell_\psi(\gamma))}{\sum_{\gamma\in P(0,T+\Delta)}\exp(\ell_\psi(\gamma))} - \frac{\sum_{\gamma \in P(0,T)}\exp(\ell_\psi(\gamma))}{\sum_{\gamma\in P(0,T+\Delta)}\exp(\ell_\psi(\gamma))}.
\end{align*}
By Proposition \ref{prop:asymp}, we have
\[
\sum_{\gamma \in P(0,T)}\exp(\ell_\psi(\gamma)) \sim \frac{e^{\PP(\psi)T}}{\PP(\psi)T},
\]
and therefore
\[
\lim_{T\to\infty}\frac{\sum_{\gamma \in P(0,T)}\exp(\ell_\psi(\gamma))}{\sum_{\gamma\in P(0,T+\Delta)}\exp(\ell_\psi(\gamma))} = e^{-\PP(\psi)\Delta}.
\]
Choosing $\Delta$ large enough so that $e^{-\PP(\psi)\Delta} \leq \frac{\epsilon}{2}$, then 
\[
\limsup_{T\to\infty}\frac{\sum_{\gamma \in Q(T,\Delta)}\exp(\ell_\psi(\gamma))}{\sum_{\gamma\in P(T,\Delta)}\exp(\ell_\psi(\gamma))} 
\geq \frac{\epsilon}{2} > 0,
\]
and thus the fact that $\phi$ is a coboundary follows from the first part of the theorem.
\end{proof}

\subsection{Proof of Theorems \ref{thm:positivepropLivshits3} and \ref{thm:positivepropLivshits4}}

We start with the following lemmas.

\begin{lemma}\label{lem:pressurebound}
       For any H\"{o}lder $\psi : M \to \R$ and any $t > 1$, we have $\PP(t\psi) < t\PP(\psi)$.
\end{lemma}
\begin{proof}
Note that for $t=1$, we have $\PP(t\psi) = t\PP(\psi)$. If $t > 1$, then Equation \eqref{eqn:pressure} yields
\[ \left.\frac{d}{ds}\right|_{s=t}\PP(s\psi) = \mu_{t\psi}(\psi) \leq \PP(\psi)-h(\mu_{t\psi}) < \PP(\psi) = \left.\frac{d}{ds}\right|_{s=t}s\PP(\psi),\]
using the fact that equilibrium states have positive entropy.
Integrating this proves the result.
\end{proof}
\begin{corollary}\label{cor:strictbound}
For any H\"{o}lder $\psi : M \to \R$, there exists $\epsilon > 0$ such that, for any $\gamma \in P$, we have
\[ \ell_\psi(\gamma) \leq (\PP(\psi)-\epsilon)\ell(\gamma)\qedhere.\]
\end{corollary}
\begin{proof}
Choose any $t > 1$. By the definition of pressure and the fact that any periodic probability measure $\delta_\gamma$ has zero entropy, for any $\gamma \in P$ we have
\[
\frac{\ell_\psi(\gamma)}{\ell(\gamma)} = \delta_\gamma(\psi) = \frac{h(\delta_\gamma) + \delta_\gamma(t\psi)}{t} \leq \frac{\PP(t\psi)}{t} < \PP(\psi). \qedhere
\]
\end{proof}

\begin{lemma} \label{lem:subexp}
    If $Q'$ grows subexponentially, then for any H\"{o}lder $\psi : M \to \R$, we have $\mu_\psi(\phi) \leq 0$.
\end{lemma}

\begin{proof}
We will show this in the diffeomorphism case, but the same argument carries over to the flow case. We have
\[
\mu_\psi(\phi) = \lim_{n\to\infty}\frac{\sum_{\gamma \in P(n)}\exp(\ell_{\psi}(\gamma))\delta_\gamma(\phi)}{\sum_{\gamma \in P(n)}\exp(\ell_{\psi}(\gamma))} \leq \lim_{n\to\infty}|\phi|_\infty\frac{\sum_{\gamma\in Q'(n)}\exp(\ell_{\psi}(\gamma))}{\sum_{\gamma \in P(n)}\exp(\ell_{\psi}(\gamma))}.
\]
Using Corollary \ref{cor:strictbound}, we can write
\[
\sum_{\gamma \in Q'(n)}\exp(\ell_\psi(\gamma)) \leq \exp((\PP(\psi)-\epsilon)n)|Q'(n)|.
\]
Since $|Q'(n)|$ 
grows subexponentially, we have
\[
\limsup_{n\to\infty}\frac{1}{n}\log \left(\sum_{\gamma \in Q'(n)}\exp(\ell_{\psi}(\gamma))\right) \leq \PP(\psi)-\epsilon < \PP(\psi) =
\lim_{n\to\infty}\frac{1}{n}\log\left( \sum_{\gamma \in P(n)}\exp(\ell_{\psi}(\gamma))\right).
\]
Therefore 
\[\lim_{n\to\infty}\frac{\sum_{\gamma\in Q'(n)}\exp(\ell_\psi(\gamma))}{\sum_{\gamma\in P(n)}\exp(\ell_\psi(\gamma))} = 0. \qedhere\]
\end{proof}
It is now straightforward to deduce that every period of $\phi$ must be nonpositive by using the well-known fact from ergodic optimization that the zero temperature limit maximizes the integral of $\phi$ among invariant probability measures.

\begin{proof}[Proofs of Theorem \ref{thm:positivepropLivshits3} and \ref{thm:positivepropLivshits4}]
    Let $\gamma \in P$ be arbitrary, and let $\delta_\gamma$ denote the $f$-invariant probability measure supported on $\gamma$. For $s > 0$, let $\mu_s$ be the equilibrium state of $s\phi$.
    For every $s$ we have
        \[ s\delta_\gamma(\phi) = h(\delta_\gamma)+\delta_\gamma(s\phi) \leq P(s\phi) = h(\mu_s) + s\mu_{s}( \phi).\]
        Therefore using Lemma \ref{lem:subexp} and the fact that all metric entropies are bounded above, 
        \[ \ell_\phi(\gamma) = \ell(\gamma)\delta_\gamma(\phi)  \leq \ell(\gamma)\limsup_{s \rightarrow \infty} \left(\frac{h(\mu_s)}{s} + \mu_{s}(\phi)\right) \leq 0.\]
        Thus every period of $\phi$ is nonpositive, and the fact that $\phi$ is cohomologous to a nonpositive function follows from the nonpositive Livshits theorem \cite{lopes1, lopes2}. 
\end{proof}

\section{Proof of Theorem \ref{thm:CCLT}}\label{sec:CLTs}

    Throughout, we assume that $\PP(\psi) > 0$, $\sigma_{\phi,\psi}^2 > 0$, and that the stable and unstable distributions of $f^t$ are not jointly integrable.
    Define the $\psi$-weighted dynamical $L$-function by
    \[ L(s,t) \coloneqq \prod_{\gamma \in P_*}\left(1-e^{\ell_\psi(\gamma)-s\ell(\gamma)+it\ell_\phi(\gamma)}\right)^{-1} = \exp \left\{ \sum_{\gamma \in P_*} \sum_{n=1}^\infty \frac{1}{n} e^{n\ell_\psi(\gamma) - sn\ell(\gamma) + itn \ell_\phi(\gamma)} \right\}.\]
    Let $L'(s,t)$ denote the derivative of $L(s,t)$ with respect to $s$, and for small enough $t$ let $s(t) \coloneqq \PP(\psi+it\phi)$. The following result was shown in the unweighted case in \cite[Section 4]{pollicottsharp}, and with some modifications to the calculations, the proof is the same in the weighted case. See also \cite[Section 6]{Petkov} for a similar weighted generalization in the case where $f^t$ is exponentially mixing.
        
    \begin{proposition} \label{prop:dolg}
        There exist $C, \rho, \delta > 0$ such that for any $t \in (-\delta, \delta)$,
        \[ \frac{L'(s,t)}{L(s,t)} + \frac{1}{s-s(t)}\]
        is analytic in $\Re(s) > \PP(\psi) - C \min\{1, |\Im(s)|^{-\rho}\}$. Furthermore, there exists $\beta > 0$, independent of $t \in (-\delta, \delta)$, such that for $\Re(s) > \PP(\psi) - C \min\{1, |\Im(s)|^{-\rho}\}$ we have
        \[ \left|\frac{L'(s,t)}{L(s,t)}\right| = O(\max\{|\Im(s)|^\beta, 1\}).\]
    \end{proposition}

    Let 
    \[ \eta(s,t) \coloneqq -\frac{L'(s,t)}{L(s,t)} = \sum_{n=1}^\infty \sum_{\gamma \in P_*} \ell(\gamma) e^{n\ell_\psi(\gamma)-s n\ell(\gamma) + itn\ell_\phi(\gamma) }.\]
    The arguments in \cite[Section 3]{pollicottsharp} can be generalized as follows.
    
    \begin{proposition} \label{prop:asymp}
        There exists $\lambda > 0$ such that 
        \begin{equation*} \label{eqn:asymp1}
        \sum_{\gamma \in P_*(0,T)} e^{\ell_\psi(\gamma)} = \frac{e^{\PP(\psi)T}}{\PP(\psi)T} \left(1 + O \left(\frac{1}{T^\lambda}\right) \right),\end{equation*}
        \begin{equation*} 
        \label{eqn:asymp2}
        \sum_{\gamma \in P_*(0,T)} \ell(\gamma) e^{\ell_\psi(\gamma)} = \frac{e^{\PP(\psi)T}}{\PP(\psi)} \left(1 + O \left(\frac{1}{T^\lambda}\right) \right).
        \end{equation*}
        The same is true replacing $P_*(0,T)$ with $P(0,T)$. 
    \end{proposition}

    \begin{proof}
    By Lemma \ref{lem:prime}, it suffices to show the asymptotics for $P_*(0,T)$.
    Define 
        \begin{align*} S_{0}(T) \coloneqq \sum_{n=1}^\infty \sum_{e^{n\ell(\gamma)} \leq T} \ell(\gamma) e^{n\ell_\psi(\gamma)},\ \ \hat{S}_{0}(T) \coloneqq \sum_{e^{\ell(\gamma)} \leq T} \ell(\gamma) e^{\ell_\psi(\gamma)},\end{align*}
        where both sums are taken over prime orbits $\gamma$. For $k \geq 1$, define
        \[ S_{k}(T) \coloneqq \int_0^T S_{k-1}(\tau)\, d\tau  = \frac{1}{k!}\sum_{n=1}^\infty \sum_{e^{n\ell(\gamma)} \leq T} \ell(\gamma) e^{n\ell_\psi(\gamma)}\left(T - e^{n \ell(\gamma)} \right)^k,\]
        where again this sum is over all prime orbits $\gamma$. We define $\hat{S}_{k}$ similarly.  
        
        Given a closed orbit $\gamma' \in P$, one can write $\gamma' = \gamma^n$ for a unique choice of $\gamma \in P_*$. Define the von Mangoldt function of this closed orbit by
        $\Lambda(\gamma') \coloneqq \ell(\gamma)$.
        Then
        \[ S_k(T) = \frac{1}{k!} \sum_{e^{\ell(\gamma')} \leq T} \Lambda(\gamma') e^{\ell_\psi(\gamma')} \left(T - e^{\ell(\gamma')} \right)^k, \]
        where the sum is over all closed orbits $\gamma'$.
        
        The following identity holds for all $d > 0$ and $k \geq 1$ \cite[p.\ 31]{ingham}:
        \begin{equation} \label{eqn:analyticnum}
        \frac{1}{2\pi i }\int_{d-i\infty}^{d+i \infty} \frac{x^s}{s(s+1) \cdots (s+k)}\,ds = \begin{dcases} 0 & \text{if } 0 < x < 1, \\ \frac{1}{k!} \left(1 - \frac{1}{x} \right)^k & \text{if }x \geq 1. \end{dcases}
        \end{equation}
        Applying this term by term to the series given by $\eta(s,0)$, we have
        \begin{equation} \label{eqn:trunacte}
            \begin{split} S_k(T) = \frac{1}{2\pi i }\int_{d-i\infty}^{d+i \infty} \frac{\eta(s,0)T^{s+k}}{s(s+1) \cdots (s+k)}\,ds.\end{split} 
        \end{equation}
        Fix $0 < \epsilon < 1/\rho$, where $\rho$ is as in Proposition \ref{prop:dolg}. Let 
        \[d \coloneqq \PP(\psi) + \frac{1}{\log(T)}, \ \ R \coloneqq (\log(T))^\epsilon, \ \ c \coloneqq \PP(\psi) - \frac{C}{2 R^\rho}.\]
        Consider the curve $\Gamma$ which is the union of the line segments $[d+iR,d+i\infty]$, $[c + iR, d + iR]$, $[c - iR, c + iR]$, $[d-iR, c - iR]$, and $[d-iR,d-i\infty]$. 
        As long as $T$ is large, Proposition \ref{prop:dolg} and the residue theorem imply
        \[ S_k(T) = \frac{T^{\PP(\psi)+k}}{\PP(\psi)(\PP(\psi)+1)\cdots(\PP(\psi)+k)} + \frac{1}{2\pi i} \int_\Gamma \frac{\eta(s,0)T^{s+k}}{s(s+1) \cdots (s+k)}\,ds.\]
        Similarly to the contour arguments in \cite{cantrellsharp, pollicottsharp}, Proposition \ref{prop:dolg} can be used to bound the integral on each segment of $\Gamma$, and one finds that for $k$ sufficiently large there exists $\alpha > 0$ such that
        \[ S_k(T) = \frac{T^{\PP(\psi)+k}}{\PP(\psi)(\PP(\psi)+1)\cdots(\PP(\psi)+k)} + O \left( \frac{T^{\PP(\psi)+k}}{(\log(T))^{\alpha}}\right).\]
        Using Lemma \ref{lem:prime}, we have
        \begin{align*}
        S_k(T) - \hat{S}_k(T) 
        &= \frac{1}{k!}\sum_{P(0,\log(T))\setminus P_*(0,\log(T))}\Lambda(\gamma)e^{\ell_\psi(\gamma)}(T-e^{\ell(\gamma)})^k \\
        &\leq \frac{1}{k!}\log(T)T^k\sum_{P(0,\log(T))\setminus P_*(0,\log(T))}e^{\ell_\psi(\gamma)} \\
        &= O(\log(T) T^{\delta+k}) \text{ for some } 0 < \delta < \PP(\psi).
        \end{align*}
        Thus for $k \geq 1$ we have 
        \[ \hat{S}_k(T) = \frac{T^{\PP(\psi)+k}}{\PP(\psi)(\PP(\psi)+1)\cdots(\PP(\psi)+k)} + O \left( \frac{T^{\PP(\psi)+k}}{(\log(T))^{\alpha}}\right).\]
        To prove the proposition, we need a similar bound for $k=0$. Since $\hat{S}_{k-1}$ is increasing and $\hat{S}_k$ is the integral of $\hat{S}_{k-1}$, for $\epsilon > 0$ we have
        \begin{equation} \label{eqn:ineq}
            \frac{1}{T\epsilon}(\hat{S}_k(T)-\hat{S}_k(T-T\epsilon)) \leq \hat{S}_{k-1}(T) \leq \frac{1}{T\epsilon}(\hat{S}_k(T+T\epsilon)-\hat{S}_k(T)).
        \end{equation}        
        Iterating this inequality $k$ times using $\epsilon = (\log T)^{-\alpha/2^j}$ at the $j$\textsuperscript{th} step, one can show 
        \[ \sum_{e^{ \ell(\gamma)} \leq T}  \ell(\gamma) e^{\ell_\psi(\gamma)} =\hat{S}_0(T) = \frac{T^{\PP(\psi)}}{\PP(\psi)} + O \left( \frac{T^{\PP(\psi)}}{(\log(T))^\lambda}\right),
        \]
        where $\lambda = \alpha/2^k$.
        By replacing $T$ with $e^{T}$, 
        this proves the second statement of the proposition:
        \[ \sum_{P_*(0,T)} \ell(\gamma) e^{\ell_\psi(\gamma)} = \frac{e^{\PP(\psi)T}}{\PP(\psi)} \left( 1 + O \left( \frac{1}{T^\lambda}\right)\right).\]
        
        As for the first statement, we use a Stieltjes integral to obtain
        \begin{align*} \sum_{e^{\ell(\gamma)} \leq T} e^{\ell_\psi(\gamma)} & = \int_2^T \frac{1}{\log(\tau)} \,d \hat{S}_0(\tau) + O(1) \\
        & = \frac{\hat{S}_0(T)}{\log(T)} + \int_2^T \frac{\hat{S}_0(\tau)}{\tau (\log(\tau))^2}\,d\tau + O(1).\end{align*}
        A standard argument using the estimate on $\hat{S}_0(T)$ then proves
        \[ \sum_{e^{ \ell(\gamma)} \leq T} e^{\ell_\psi(\gamma)} = \frac{T^{\PP(\psi)}}{\PP(\psi)\log(T)} + O \left(\frac{T^{\PP(\psi)}}{(\log(T))^{1+\lambda}} \right), \]
        or equivalently,
        \[ \sum_{\gamma \in P_*(0,T)} e^{\ell_\psi(\gamma)} = \frac{e^{\PP(\psi)T}}{\PP(\psi)T} \left( 1 + O \left(\frac{1}{T^\lambda}\right) \right).\qedhere \]
    \end{proof}

    To obtain the desired central limit theorem, we modify the functions $S_k$ and $\hat{S}_k$ to keep track of the observable $\phi$. For $t \in (-\delta, \delta)$, define
    \begin{align*} S_{0,t}(T) \coloneqq \sum_{n=1}^\infty \sum_{e^{n\ell(\gamma)} \leq T} \ell(\gamma) e^{n\ell_\psi(\gamma) + itn \ell_\phi(\gamma)}, & & \hat{S}_{0,t}(T) = \sum_{e^{\ell(\gamma)} \leq T} \ell(\gamma) e^{\ell_\psi(\gamma) + it\ell_\phi(\gamma)},\end{align*}
    where the sums are taken over prime orbits $\gamma$. For $k \geq 1$, define
    \[ S_{k,t}(T) \coloneqq \int_0^T S_{k-1,t}(\tau)\, d\tau  = \frac{1}{k!}\sum_{n=1}^\infty \sum_{e^{n\ell(\gamma)} \leq T} \ell(\gamma) e^{n\ell_\psi(\gamma)+itn\ell_\phi(\gamma)}\left(T - e^{n \ell(\gamma)} \right)^k,\]
    and define $\hat{S}_{k,t}$ similarly.
    The goal is to repeat the arguments from Proposition \ref{prop:asymp} with $S_{k,t}$ in place of $S_k$. However, two issues arise when
    considering nonzero $t$: first, as ${t \in (-\delta, \delta)}$ varies, the singularity $s(t)$ leaves the region inside the contour used in Proposition \ref{prop:asymp}.
    Second, if $t \neq 0$, then it no longer makes sense to say that $\hat{S}_{k,t}(T)$ is increasing. Both issues were handled in \cite{cantrellsharp} when $\psi \equiv 0$; we will now provide these arguments for $\psi \not\equiv 0$.

    Applying \eqref{eqn:analyticnum} term by term to $\eta(s,t)$, we have
    \[ S_{k,t}(T) = \frac{1}{k!}\sum_{n=1}^\infty \sum_{e^{n\ell(\gamma)} \leq T} \ell(\gamma) e^{n\ell_\psi(\gamma)+itn\ell_\phi(\gamma)} (T - e^{n\ell(\gamma)})^k = \frac{1}{2\pi i} \int_{d-i\infty}^{d+i \infty} \frac{\eta(s,t) T^{s+k}}{s(s+1) \cdots (s+k)}\,ds.\]
    Define $d, R, c, \Gamma$ as in the proof of Proposition \ref{prop:asymp}, except in the definition of $\Gamma$ one replaces $[c-iR,c+iR]$ with the segments $[c-iR,c-ir_2]$, $[c-ir_2,r_1-ir_2]$, $[r_1-ir_2,r_1+ir_2]$, $[r_1+ir_2,c+ir_2]$, $[c+ir_2,c+iR]$, where $0<r_1<\PP(\psi)$, $r_2> 0$ are chosen so that, for small enough $t$, the contour now encompasses the singularities $s(t)$ while remaining in the region described by Proposition \ref{prop:dolg}. Thus we still have
    \[S_{k,t}(T) = \frac{T^{s(t)+k}}{s(t)(s(t)+1) \cdots (s(t)+k)} + \frac{1}{2\pi i} \int_\Gamma \frac{\eta(s,t) T^{s+k}}{s(s+1) \cdots (s+k)}\,ds,\]
    and modifying the estimates in \cite[Lemma 4.1]{cantrellsharp} 
    we obtain $k \geq 1$ and $\alpha > 0$ such that 
    \[  S_{k,t}(T) = \frac{T^{s(t)+k}}{s(t)(s(t)+1) \cdots (s(t)+k)} + O \left( \frac{T^{\PP(\psi)+k}}{(\log(T))^\alpha}\right),\]
    where the error term is independent of $t \in (-\delta, \delta)$.
    Using Lemma \ref{lem:prime}, we equivalently have
    \begin{equation} \label{eqn:obsn}
    \hat{S}_{k,t}(T) = \frac{T^{s(t)+k}}{s(t)(s(t)+1) \cdots (s(t)+k)} + O \left( \frac{T^{\PP(\psi)+k}}{(\log(T))^\alpha}\right).
    \end{equation}  
    
    We now follow the calculations in \cite[Section 5]{cantrellsharp} to get an estimate on $\hat{S}_{0,t}(T)$. Let ${\epsilon \coloneqq (\log(T))^{-\alpha/2}}$. Notice that decreasing $\alpha$ preserves \eqref{eqn:obsn}, so we may assume that $\alpha < 2 \lambda$, where $\lambda$ is as in Proposition \ref{prop:asymp}. Using \eqref{eqn:obsn}, we have on the one hand
    \[ \hat{S}_{k,t}(T+T\epsilon) - \hat{S}_{k,t}(T) = \frac{T^{s(t)+k} \epsilon}{s(t) (s(t)+1) \cdots (s(t)+k-1)} + O \left( \frac{T^{\PP(\psi)+k}}{(\log(T))^\alpha} \right), \]
    while on the other,
    \begin{align*}
        \hat{S}_{k,t}(T+T\epsilon) - \hat{S}_{k,t}(T) & = \frac{1}{k!}\sum_{T< e^{\ell(\gamma)} \leq T+T\epsilon} \ell(\gamma) e^{ \ell_\psi(\gamma)+it \ell_\phi(\gamma) } \left( T - e^{\ell(\gamma)}\right)^k \\
        & +  \frac{T\epsilon}{(k-1)!} \sum_{e^{\ell(\gamma)} \leq T+T\epsilon} \ell(\gamma) e^{\ell_\psi(\gamma)+it \ell_\phi(\gamma)} \left( T - e^{\ell(\gamma)}\right)^{k-1} \\
        & + \frac{1}{k!}\sum_{j=2}^k (T\epsilon)^j \binom{k}{j} \sum_{e^{\ell(\gamma)} \leq T+T\epsilon}\ell(\gamma) e^{\ell_\psi(\gamma)+it \ell_\phi(\gamma)} \left( T - e^{\ell(\gamma)}\right)^{k-j}.
    \end{align*}
    Using Proposition \ref{prop:asymp}, we observe that
    \begin{align*} \sum_{T < e^{\ell(\gamma)} \leq T+T\epsilon} \ell(\gamma) e^{\ell_\psi(\gamma)} &= \epsilon T^{\PP(\psi)}\frac{(1+\epsilon)^{\PP(\psi)}-1}{\epsilon\PP(\psi)}+O(T^{\PP(\psi)}(\log T)^{-\lambda}) \\
    &= \epsilon T^{\PP(\psi)}\left(1+o(1)+O((\log T)^{\alpha/2-\lambda})\right).
    \end{align*}
    Since $\alpha/2 < \lambda$, we have
    \[ \sum_{T < e^{\ell(\gamma)} \leq T + T\epsilon} \ell(\gamma) e^{it \ell_\phi(\gamma) + \ell_\psi(\gamma)} \left( T - e^{\ell(\gamma)}\right)^k = O(T^{\PP(\psi)+k} \epsilon^{k+1}).\]
    Note this also gives us 
    \[ \frac{T \epsilon}{(k-1)!} \sum_{e^{\ell(\gamma)} \leq T + T\epsilon} \ell(\gamma) e^{it \ell_\phi(\gamma) + \ell_\psi(\gamma)} \left( T - e^{\ell(\gamma)}\right)^{k-1} = T \epsilon \hat{S}_{k-1,t}(T) + O(T^{\PP(\psi)+k} \epsilon^{k+1}). \]
    Finally, it is clear that
    \[ \frac{1}{k!}\sum_{j=2}^k (T \epsilon)^j \binom{k}{j} \sum_{e^{\ell(\gamma)} \leq T + T\epsilon}\ell(\gamma) e^{it \ell_\phi(\gamma) + \ell_\psi(\gamma)} \left( T - e^{\ell(\gamma)}\right)^{k-j} = \sum_{j=2}^k O(T^{\PP(\psi)+k} \epsilon^j).\]
    Putting all of the above together, for $k \geq 1$ we have
    \[ T \epsilon \hat{S}_{k-1,t}(T) = \frac{T^{s(t)+k} \epsilon}{s(t) (s(t)+1) \cdots (s(t)+k-1)} + O \left( \frac{T^{\PP(\psi)+k}}{(\log(T))^\alpha} \right) + O(T^{\PP(\psi)+k}\epsilon^2 ). \]
    Dividing both sides by $T\epsilon$ yields
    \[ \hat{S}_{k-1,t}(T) = \frac{T^{s(t)+k-1} }{s(t) (s(t)+1) \cdots (s(t)+k-1)} + O \left( \frac{T^{\PP(\psi)+k-1}}{(\log(T))^{\alpha/2}}\right),\]
    which is the desired inductive step. Iterating this $k$ times, we obtain
    \[
    \hat{S}_{0,t}(T) = \frac{T^{s(t)}}{s(t)} + O\left(\frac{T^{\PP(\psi)}}{(\log(T))^{\nu}}\right),
    \]
    or in other words,
    \begin{equation}\label{eqn:yay}
    \sum_{\gamma \in P_*(0,T)} \ell(\gamma) e^{\ell_\psi(\gamma)+it \ell_\phi(\gamma)} 
    = \frac{e^{Ts(t)}}{s(t)} + O \left( \frac{e^{T\PP(\psi)}}{T^\nu} \right),\end{equation}
    where $\nu \coloneqq \alpha/2^k$ and the error term is independent of $t \in (-\delta, \delta)$.
    
    By \eqref{eqn:pressuretaylor}, recall that for any fixed $t \in \R$,
    \[ Ts(t/\sqrt{T}) = \PP(\psi)T - \frac{\sigma_{\phi,\psi}^2 t^2}{2} + o(1).\]
    Replacing $t$ with $t/\sqrt{T}$ in \eqref{eqn:yay} and using the fact that the error term is independent of $t$, 
    \begin{equation}\label{eqn:yay1}
    g(T) \coloneqq \sum_{\gamma \in P_*(0,T)} \ell(\gamma)e^{ \ell_\psi(\gamma)+ it\ell_\phi(\gamma)/\sqrt{T}} \sim
    \frac{e^{\PP(\psi)T-\sigma^2_{\phi,\psi}t^2/2}}{\PP(\psi)}.
    \end{equation}
    We would like an asymptotic similar to this, except with the $\ell(\gamma)$ terms removed.
    Similarly to the proof of Proposition \ref{prop:asymp}, we start by writing
    \begin{align*}
    \sum_{\gamma \in P_*(0,T)} e^{\ell_\psi(\gamma)+it \ell_\phi(\gamma)/\sqrt{\ell(\gamma)} } & = \int_1^T \frac{dg(\tau)}{\tau} + O(1) \\
    &= \frac{g(T)}{T} + O \left( \frac{e^{\PP(\psi)T}}{T^2} \right)+O(1).
    \end{align*}
    Therefore by \eqref{eqn:yay1},
    \begin{equation} \label{eqn:yay?}
    \sum_{\gamma \in P_*(0,T)} e^{\ell_\psi(\gamma)+it \ell_\phi(\gamma)/\sqrt{\ell(\gamma)}} \sim \frac{e^{\PP(\psi)T - \sigma_{\phi,\psi}^2t^2/2}}{\PP(\psi)T}.
    \end{equation}
    Next, we need to replace $P_*(0,T)$ with $P_*(T,\Delta)$. Using $P_*(T,\Delta) = P_*(0,T+\Delta)-P_*(0,T)$ along with \eqref{eqn:yay?}, it is easy to deduce
    \[
    \sum_{\gamma \in P_*(T,\Delta)} e^{\ell_\psi(\gamma)+it \ell_\phi(\gamma)/\sqrt{\ell(\gamma)}}
    \sim (e^{\PP(\psi)\Delta}-1)\frac{e^{\PP(\psi)T-\sigma^2_{\phi,\psi}t^2/2}}{\PP(\psi)T}.
    \]
    Finally, we need to replace $\sqrt{\ell(\gamma)}$ with $\sqrt{T}$. Note that for a fixed $t$ we have 
    \[ \sum_{\gamma \in P_*(T,\Delta)} \left|e^{\ell_\psi(\gamma)+it\ell_\phi(\gamma)/\sqrt{\ell(\gamma)}} - e^{\ell_\psi(\gamma)+it\ell_\phi(\gamma)/\sqrt{T}} \right| = O \left( \frac{1}{\sqrt{T}} \sum_{\gamma \in P_*(T,\Delta)} e^{\ell_\psi(\gamma)} \right) = O\left(\frac{e^{\PP(\psi)T}}{T^{3/2}}\right),\] 
    thus
    \begin{equation}\label{eqn:yay3}
    \sum_{\gamma \in P_*(T,\Delta)} e^{\ell_\psi(\gamma)+it \ell_\phi(\gamma)/\sqrt{T}}
    \sim (e^{\PP(\psi)\Delta}-1)\frac{e^{\PP(\psi)T-\sigma^2_{\phi,\psi}t^2/2}}{\PP(\psi)T}.
    \end{equation}
    Comparing \eqref{eqn:yay3} to the same asymptotic for $t=0$, we deduce
    \[
    \lim_{T \rightarrow \infty} \frac{\sum_{\gamma \in P_*(T,\Delta)} e^{\ell_\psi(\gamma) + it\ell_\phi(\gamma)/\sqrt{T}}}{ \sum_{\gamma \in P_*(T,\Delta)} e^{\ell_\psi(\gamma)}} = e^{-\sigma_{\phi,\psi}^2t^2/2}.
    \]
    Similarly, using \eqref{eqn:yay1} and \eqref{eqn:yay?},
    \[
    \lim_{T \rightarrow \infty} \frac{\sum_{\gamma \in P_*(0,T)} \ell(\gamma)e^{\ell_\psi(\gamma) + it\ell_\phi(\gamma)/\sqrt{T} }}{ \sum_{\gamma \in P_*(0,T)} \ell(\gamma)e^{\ell_\psi(\gamma)}} = \lim_{T \rightarrow \infty} \frac{ \sum_{\gamma \in P_*(0,T)} e^{\ell_\psi(\gamma) + it\ell_\phi(\gamma)/\sqrt{\ell(\gamma)} }}{ \sum_{\gamma \in P_*(0,T)} e^{\ell_\psi(\gamma)}} = e^{-\sigma_{\phi,\psi}^2t^2/2},
    \]
    and one can use Lemma \ref{lem:prime} to replace $P_*$ with $P$ in any of these limits.
    By L{\'e}vy's continuity theorem, each of these is its own central limit theorem, and this includes Theorem \ref{thm:CCLT}.

\end{document}